\documentclass[letterpaper, 10 pt, conference]{ieeeconf}  

\IEEEoverridecommandlockouts                              
\usepackage{xcolor}
\usepackage{float}
\usepackage[fleqn]{amsmath}
\let\proof\relax 

\usepackage{amsthm,amssymb,graphicx,url}
\let\labelindent\relax 
\usepackage{enumitem}
\usepackage{booktabs}
\usepackage{subcaption}
 
\usepackage{appendix}

\usepackage{wrapfig}
\usepackage{hyperref} %
\usepackage{soul}  
\usepackage{cleveref}
\usepackage{bm}
\usepackage{multicol}
\usepackage{mathtools}
\usepackage{cite}

\crefname{equation}{}{} 
\crefname{assumption}{Assumption}{}
\crefname{table}{Table}{} 
\crefname{figure}{Fig.}{}
\crefname{section}{Section}{}
\crefname{remark}{Remark}{}
\usepackage{algorithm,algorithmicx} 
\usepackage{algpseudocode} 
\newlength\myindent
\DeclareMathOperator*{\argmin}{argmin}

\DeclarePairedDelimiter\abs{\lvert}{\rvert}
\newcommand{\norm}[1]{\left\lVert#1\right\rVert}

\newcommand{\infnorm}[1]{\left\lVert#1\right\rVert_\infty}

\def\tilx{\tilde{x}}
\def\dotx{\dot{x}}

\def\mbR{\mathbb{R}}

\def \mbZ{\mathbb{Z}}

\def\mcC{\mathcal{C}}

\def \hatx{\hat{x}}

\def\trieq{\triangleq}

\newtheorem{theorem}{Theorem}
\newtheorem{lemma}{Lemma}
\theoremstyle{definition}  \newtheorem{defn}{Definition}
\theoremstyle{definition} \newtheorem{assumption}{Assumption}
\theoremstyle{remark}  \newtheorem{remark}{Remark}

\def\hd{{\hat d}}
\def\mbZ{\mathbb{Z}}

\def\mcC{{\mathcal{C}}}

\def \mcK{{\mathcal{K}}}
\def \bbracket#1{\bm{[}#1\bm{]}}

\def \proof{\noindent{\it Proof}. }

\def \QP{\textbf{QP}}
\def \RaQP{\textbf{aR-QP}}

\usepackage{commath}

\title{\LARGE \bf
Adaptive Robust Quadratic Programs using Control Lyapunov \\and Barrier Functions}
\author{Pan Zhao, Yanbing Mao, Chuyuan Tao, Naira Hovakimyan, Xiaofeng Wang 
\thanks{This work is supported by AFOSR, NASA and NSF NRI grant \#1830639. }
\thanks{P. Zhao, Y. Mao, C. Tao and N. Hovakimyan are with the Department of 
Mechanical Science and Engineering, University of
Illinois at Urbana-Champaign, Urbana, IL 61801, USA. {\tt\small panzhao2, ybmao, chuyuan2, nhovakim@illinois.edu} }
\thanks{X. Wang is with the Department of Electrical Engineering, University of South Carolina, Columbia,
SC 29208, USA.
        {\tt\small wangxi@cec.sc.edu}}%
}

\begin{document}

\maketitle
\thispagestyle{empty}
\pagestyle{empty}

\begin{abstract}
This paper presents adaptive robust quadratic program (QP) based control using control Lyapunov and barrier functions for nonlinear systems subject to time-varying and state-dependent uncertainties. 
An adaptive estimation law is proposed to estimate the pointwise value of the uncertainties with pre-computable estimation error bounds. The estimated uncertainty and the error bounds are then used to formulate a robust QP, which ensures that the actual uncertain system will not violate the safety constraints defined by the control barrier function. Additionally, the accuracy of the uncertainty estimation can be systematically improved by reducing the estimation sampling time, leading subsequently to reduced conservatism of the formulated robust QP. The proposed approach is validated in simulations on an adaptive cruise control problem and through comparisons with existing approaches.  
\end{abstract}


\section{INTRODUCTION}

Control Lyapunov functions (CLFs) provide a powerful approach to analyze the closed-loop stability and synthesize stabilizing control signals for nonlinear systems  without resorting to an explicit feedback control law, \cite{sontag1983clf,Artstein1983clf}. They also facilitate optimization based control, e.g. via Quadratic Programs (QPs) \cite{galloway2015torque}, which could explicitly consider input constraints. 
On the other hand, inspired by the barrier functions that are used to certify the forward invariance of a set, control barrier functions (CBFs) are proposed to design feedback control laws to ensure that the system states stay in a safe set  \cite{wieland2007cbf}. 
 Unification of CLF and CBF conditions into a single QP was studied in \cite{ames2016cbf-tac}, which allows compromising the CLF-defined control objectives to enforce safety.  

Due to reliance on dynamic models, the performance of CLF and/or CBF based control  is deteriorted in the presence of model uncertainties and disturbances. As an example, the safety constraints defined using CBF and a nominal model may be violated in the presence of uncertainties and disturbances. The paper \cite{xu2015robustness} studied the robustness of CBF based control in the presence of bounded disturbance and established formal bounds on violation of the CBF constraints. Input-to-state safety in the presence of input disturbance was studied in \cite{kolathaya2018inputSafety} to ensure system states stay in a set that is close to the original safe set. However, in practice, safety constraints may often need to be strictly enforced. Towards this end, adaptive CBF approaches were proposed in \cite{taylor2019adaptiveSafety,lopez2020ra-cbf} for systems with parametric uncertainties. Robust CLF and CBF based control was explored in \cite{nguyen2016optimalACC}, where the state-dependent uncertainties  were assumed to have known uncertainty bounds that were used to formulate robust constraints. This approach can be rather conservative, as the uncertainty bounds need to hold in the entire set of admissible states, which can be overly large. Bayesian learning based approaches using Gaussian process regression (GPR) were also proposed to learn state-dependent uncertainties and subsequently use the learned model to enforce probabilistic safety constraints \cite{wang2018safe,khojasteh2020probabilistic,takano2020robustcbf}. Nevertheless, the expensive computation associated with GPR prohibits the use of these approaches for real-time control; additionally, the predicted uncertainty for a state far away from the collected data points can be quite poor, leading to overly conservative performance or even infeasible QP problems.  


 We present an adaptive estimation based approach to design of CLF and CBF based controllers via QPs in the presence of time-varying and state-dependent nonlinear uncertainties, while strictly enforcing the safety constraints. With the proposed estimation scheme, the pointwise value of the uncertainties can be estimated with pre-computable error bounds. The estimated uncertainty and the error bounds are then used to formulate a QP with robust constraints. We also show that by reducing the estimation sampling time,  the estimated uncertainty can be arbitrarily accurate after a {\it single} sampling interval, which implies that the conservatism of the robust constraints can be arbitrarily small. The effectiveness of the approach is demonstrated on an automotive cruise control (ACC) problem in simulations.

Notations: The symbol $\mathbb{R}^n$ denotes the $n$-dimensional real vector space.  The notations 
 $\mbZ_i$ and $\mbZ_1^n$ denote the integer sets $\{i, i+1, i+2, \cdots\}$ and $\{1, 2,\cdots,n\}$, respectively. 
The notations $\norm{\cdot}$  and $\infnorm{\cdot}$ denote the $2$-norm  and $\infty$-norm of a vector (or a matrix), respectively. 


\section{Preliminaries}\label{sec:preliminaries}
Consider a nonlinear control-affine system 
\begin{equation}\label{eq:dynamics}
    \dot{x} = f(x)+g(x)u + d(t,x),\quad x(0)=x_0,
\end{equation}
where $x\in X \subset \mbR^n$, $u \in U \subset \mbR^m$, $f:\mbR^n\rightarrow \mbR^n$ and $g:\mbR^n\rightarrow \mbR^m$ are known and locally Lipschitz continuous functions, $d(t,x)$ represents the   time-varying and state-dependent uncertainties that may include parameteric uncertainties and external disturbances.
Suppose $X$ is a compact set, and  the control constraint set $U$ is defined as $U\trieq \{u\in \mbR^m: \underline{u}\leq u \leq \bar{u}\}$, where $\underline{u},\bar{u}\in\mbR^m$ denote the lower  and upper bounds of all control channels, respectively. Hereafter, we sometimes write $d(t,x)$ as $d$ for brevity. 
\begin{assumption}\label{assump:lipschitz-bound-fg}
There exist positive constants $l_d,\ l_t,\ b_d $ such that for any $x,y \in X$ and $t,\tau\geq 0$, the following inequalities hold:
\begin{align} 
\left\| {d(t,x) - d(\tau ,y)} \right\| &\le {l_d}\left\| {x - y} \right\| + {l_t}\abs{t-\tau}, \label{eq:tilf-lipschitz-cond}\\
\left\| {d(t,0)} \right\| & \le {b_d}.  \label{eq:tilf-tilg-x0-bound}
\end{align}
Moreover, the constants $l_d,\ l_t$ and $\ b_d $ are known. 
\end{assumption}

\begin{remark}
This assumption essentially indicates that the growth rate of the uncertainty $d$ is bounded so that it can be estimated by the estimation law (to be presented in Section~\ref{sec:adaptive-estimate})  with quantifiable error bounds. 
\end{remark}



\begin{lemma}\label{lemma:d-bound}
Given Assumption~\ref{assump:lipschitz-bound-fg}, for any $t\geq 0$,  $x\in X$ and $u\in U$, we have 
\begin{equation}\label{eq:d-xdot-bound}
    \norm{d(t,x)}\leq \theta, \quad \norm{\dot{x}(t)}\leq \phi, 
\end{equation}
where
{\setlength{\mathindent}{1mm}
\begin{align}
    \theta &\trieq l_d \max_{x\in X} \norm{x} + b_{d}, \label{eq:theta-defn}\\
    \phi &\trieq  \max_{x\in X, u\in U}\norm{f(x)+g(x)u}+\theta. \label{eq:phi-defn}
\end{align}}
\end{lemma}
\proof 
Note that  
\begin{IEEEeqnarray*}{rl}
   \norm{d(t,x)} & =   \norm{d(t,x) - d(t,0) + d(t,0)}  \\
   & \leq l_d \norm{x} +  \norm{d(t,0)}    \leq  l_d \norm{x} + b_d = \theta,
\end{IEEEeqnarray*}
where the two inequalities hold due to \cref{eq:tilf-lipschitz-cond} and \cref{eq:tilf-tilg-x0-bound} in Assumption~\ref{assump:lipschitz-bound-fg}. Therefore, the dynamics  \eqref{eq:dynamics} implies that 
\begin{IEEEeqnarray*}{rl}
   \norm{\dotx}&\leq \norm{f(x)+g(x)u} + \norm{d(t,x)} \leq \phi.
\end{IEEEeqnarray*}
The proof is complete. \qed
\subsection{Control Lyapunov Function}
Control Lyapunov function (CLF) provides a way to analyze the closed-loop stability and synthesize a stabilizing control signal without constructing an explicit control law  \cite{sontag1983clf,Artstein1983clf}.  A formal definition of CLF is given as follows.
\begin{defn}\label{defn:clf}
A continuously differentiable  function $V:
\mbR^n\rightarrow\mbR$ is a CLF for the system \eqref{eq:dynamics}, if it is positive definite and there exists a class $\mcK$\footnote{A function $\alpha:[0,a)\rightarrow[0,\infty)$ is said to belong to {\it class} $\mcK$ for some $a>0$, if it is strictly increasing and $\alpha(0)=0$.} function $\alpha(\cdot)$ such that 
{\setlength{\mathindent}{0mm}
\begin{align}
   \inf_{u\in U}\{L_f V(x)+ L_gV(x)u+V_x(x)d\} \leq -\alpha(V(x)), \label{eq:V-CLF-cond-inf}  
\end{align}}for all $t\geq 0$ and all $ x\in X$, 
where $V_x(x) \trieq \frac{\partial  V(x)}{\partial x}$, $L_fV(x)\triangleq \frac{\partial V(x)}{\partial x}f(x)$ and $L_gV(x)\triangleq \frac{\partial V(x)}{\partial x}g(x)$. 
\end{defn}Definition~\ref{defn:clf} allows us to consider the set of all stabilizing control signals for every point $x\in X$ and $t\geq 0$:
\begin{multline}\label{eq:kclf-defn}
    K_\textup{clf}(t,x)\triangleq \{u\in U: L_f V(x)+ L_gV(x)u+V_x(x)d(t,x) \\
    \leq -\alpha(V(x)) \}. 
\end{multline}
\subsection{Control Barrier Function}
CBFs are introduced to ensure {\it forward invariance} (often termed as {\it safety} in the literature) of a set, defined as some superlevel set of a function: 
$
    \mathcal{C}\trieq \{x\in X: h(x)\geq 0\},
$
where $h:\mbR^n\rightarrow \mbR$ is a continuously differentiable function. A formal definition of CBF is stated as follows \cite{ames2016cbf-tac}. 
\begin{defn}\label{defn:cbf} (CBF\cite{ames2016cbf-tac})
A continuously differentiable function  $h:\mbR^n\rightarrow \mbR$ is a (zeroing) CBF for the system \eqref{eq:dynamics}, if there exists an extended class $\mcK$ function\footnote{A function $\beta:(-b,a) \rightarrow (-\infty, \infty)$ is said to belong to {\it extended class} $\mcK$ for some $a,b>0$, if it is strictly increasing and $\beta(0)=0$.}  $\beta(\cdot)$ such that 
{\setlength{\mathindent}{0mm}
\begin{equation}\label{eq:h-CBF-cond-inf}
    \sup_{u\in U}\{L_fh(x)+L_gh(x)u+  h_x(x)d\} \geq -\beta(h(x)),\quad
\end{equation}}for all $x\in X$ and all $t\geq 0$, where $h_x(x)\triangleq\frac{\partial h(x)}{\partial x}$, $L_fh(x)\triangleq \frac{\partial h(x)}{\partial x}f(x)$ and $L_gh(x)\triangleq \frac{\partial h(x)}{\partial x}g(x)$.
\end{defn}
The existence of a CBF satisfying \eqref{eq:h-CBF-cond-inf} ensures that if $x(0)\in \mcC$, i.e. $h(x)\geq 0$, then there exists a control law $u(t)\in U$ such that for all $t\geq 0$, $x(t)\in \mcC$. Similarly, Definition~\ref{defn:cbf} allows us  to consider all control signals for each $x\in X$ and $t\geq 0$ that render $\mcC$ forward invariant:
{
\begin{multline}\label{eq:kcbf-defn}
    K_\textup{cbf}(t,x)\triangleq \{u\in U:  L_fh(x)+L_gh(x)u+  h_x(x)d \\
    \geq -\beta(h(x))\}. 
\end{multline}}
\begin{remark}
We consider CBFs with relative degree 1 in this work, i.e., assuming $L_gh(x) \not\equiv 0$ in \eqref{eq:h-CBF-cond-inf}; extension to the high relative-degree cases \cite[Section III.C]{ames2016cbf-tac} will be addressed as future work.  
\end{remark}
\begin{remark}
Definitions~\ref{defn:clf} and \ref{defn:cbf} utilize the true uncertainty $d$, which is not accessible in real applications. Therefore, it is impossible to verify whether a given function is a CLF (CBF) according to Definition~\ref{defn:clf} (Definition~\ref{defn:cbf}). 
\end{remark}
One way to resolve this issue is to derive a sufficient and {\it verifiable} condition for \eqref{eq:V-CLF-cond-inf} or \eqref{eq:h-CBF-cond-inf} using the worst-case bound on the uncertainty $d$ in \eqref{eq:d-xdot-bound}. The following lemma gives such conditions. The proof is straightforward considering the bound on $d$ in \eqref{eq:d-xdot-bound} and subsequently the inequalities $V_x(x)d\leq \norm{V_x(x)}\theta$ and $h_x(x)d\leq \norm{h_x(x)}\theta$. 
\begin{lemma}\label{lem:clf-cbf-verifiable-defn}
A continuously differentiable  function $V:
\mbR^n\rightarrow\mbR$ is a robust CLF for the uncertain system \eqref{eq:dynamics} under Assumption~\ref{assump:lipschitz-bound-fg}, if it is positive definite and there exists a class $\mcK$ function $\alpha(\cdot)$ such that 
{\setlength{\mathindent}{0mm}
\begin{align}
   \inf_{u\in U}\{L_f V(x)+ L_gV(x)u+\norm{V_x(x)}\theta\} \leq -\alpha(V(x)), \label{eq:V-CLF-cond-inf-R}  
\end{align}}for all $t\geq 0$ and all $ x\in X$. Similarly, a continuously differentiable function  $h:\mbR^n\rightarrow \mbR$ is a robust (zeroing) CBF for the uncertain system \eqref{eq:dynamics} under Assumption~\ref{assump:lipschitz-bound-fg}, if there exists an extended class $\mcK$ function  $\beta(\cdot)$ such that 
{\setlength{\mathindent}{0mm}
\begin{equation}\label{eq:h-CBF-cond-inf-R}
    \sup_{u\in U}\{L_fh(x)+L_gh(x)u- \norm{h_x(x)}\theta\} \geq -\beta(h(x)),\quad
\end{equation}}for all $x\in X$ and all $t\geq 0$. 
\end{lemma}

\subsection{Standard QP Formulation}
The admissible sets of control signals defined in \eqref{eq:kclf-defn} and \eqref{eq:kcbf-defn} inspire optimization based control. Recent work shows that CLF and CBF conditions can be unified into a QP \cite{ames2016cbf-tac}:
{\setlength{\mathindent}{0cm}
\setlength{\belowdisplayskip}{0pt} \setlength{\belowdisplayshortskip}{0pt}
\begin{align}
    \hline
    & u^\star(t,x)  =  \argmin_{(u,\delta)\in \mbR^{m+1}} \frac{1}{2}u^TH(x)u + p\delta^2  \quad  (\textup{\QP}) \nonumber \\ 
     \textup{s.t. } &
     \small
    L_fV(x)+L_gV(x)u + V_x(x)d +\alpha(V(x))< \delta, \label{eq:V-complete-inequality}\\
    & \small{L_fh(x)+L_gh(x)u+ h_x(x) d+\beta(h(x))}>0,\label{eq:h-complete-inequality}\\ 
   & u \in U, \\
    \hline \nonumber
\end{align}}where $H(x)$ is a (pointwise) positive definite matrix and $\delta$ is a positive constant to relax the CLF constraint that is penalized by $p>0$. In this formulation, the CBF condition is often associated with safety and therefore is imposed as a {\it hard} constraint. In contrast, the CLF constraint is related to control objective (e.g. tracking a reference) and could be relaxed to ensure the feasiblity of the QP when safety is a major concern. Therefore, it is imposed as a {\it soft} constraint.


\begin{remark}\label{remark:qp-not-implementable}
The constraints \eqref{eq:V-complete-inequality} and \eqref{eq:h-complete-inequality} depend on the true uncertainty $d$, which makes (\QP) not implementable in practice. Although the worst-case bound \eqref{eq:d-xdot-bound} can be used to derive robust versions of  \eqref{eq:V-complete-inequality} and \eqref{eq:h-complete-inequality} as done in \cite{nguyen2016optimalACC}, the resulting constraints are independent of the {\it actual} uncertainty or disturbance and thus can be overly conservative, as shown in Section~\ref{sec:simulation}. 
\end{remark}

\section{ Adaptive Robust QP Control using CLBFs}\label{sec:ad-CLBF QP}
In this section, we first introduce an adaptive estimation scheme to estimate the uncertainty with pre-computable error bounds, which can be systematically improved by reducing the  estimation sampling time. We then show how the estimated uncertainty, as well as the error bounds, can be used to formulate a robust QP, while the introduced conservatism  can be  arbitrarily reduced, subject to only hardware limitations.


\subsection{Adaptive Estimation of the Uncertainty}\label{sec:adaptive-estimate}
We extend the piecewise-constant adaptive (PWCA) law in \cite[Section 3.3]{naira2010l1book} to estimate the pointwise value of $d(t,x)$ at each time instant $t$. Similar results are available in \cite{wang2017adaptiveMPC} which considers control non-affine dynamics under more sophisticated assumptions about the uncertainty, and in \cite{zhao2020lpv} where the nominal dynamics is described by a linear parameter-varying model.
The PWCA law consists of two elements, i.e., a state predictor and an adaptive law, which are explained as follows. The state predictor is defined as:
\begin{equation}\label{eq:state-predictor}
\begin{split}
    \dot{\hatx} & =  f(x) +  g(x)u + \hd(t) -a\tilx,\quad \hatx(0)= x_0 ,
\end{split}
\end{equation}
where $\tilx \triangleq \hatx - x$ is the prediction error,  $a$ is an arbitrary positive constant. A discussion about the role of $a$ available in \cite{zhao2020lpv}. The adaptive estimation, $\hd(t)$, is updated in a piecewise-constant way:
\begin{equation}\label{eq:adaptive_law}
\left\{
\begin{split}
   \hd(t)
    &= 
    \hd(iT),  \quad t\in [iT, (i+1)T),\\
    \hd(iT)
    &= - \frac{a}{e^{aT}-1}\tilde{x}(iT),  
\end{split}\right.
\end{equation}
where $T$ is the estimation sampling time, and $i=0,1,2,\cdots$. 
\begin{remark}
The PWCA law does not estimate the analytic expression of the uncertainty; instead it estimates the {\it pointwise value} of the uncertainty at each time instant. As will be shown in Section~\ref{sec:sub-aR-QP}, the estimated uncertainty can be used to formulate a QP with robust constraints. 
\end{remark}

Let us first define: 
\begin{align}
    \gamma(T) &\triangleq 2\sqrt{n}\eta T + \sqrt{n}(1- e^{-aT})\theta, \label{eq:gammaTs-defn} \\
    \eta & \trieq l_t+l_d\phi, \label{eq:beta-defn}
\end{align}
where $\theta$ and $\phi$ are defined in \eqref{eq:theta-defn} and \eqref{eq:phi-defn}, respectively. We next establish the estimation error bounds associated with the estimation scheme in \eqref{eq:state-predictor} and \eqref{eq:adaptive_law}.
\begin{lemma}\label{lemma:estiamte-error-bound}
Given the dynamics \eqref{eq:dynamics}, and the estimation law in \eqref{eq:state-predictor} and \eqref{eq:adaptive_law}, subject to Assumption~\ref{assump:lipschitz-bound-fg}, the estimation error can be bounded as 
\begin{equation}\label{eq:estimation_error_bound}
     \norm{\hd(t)-d(t,x(t))} \leq
     \left\{
    \begin{array}{ll}
    \theta,\quad &\forall~ 0\leq t<T, \\
    \gamma(T), \quad &\forall~ t\geq T,
    \end{array}
    \right.
\end{equation}
Moreover, 
$
    \lim_{T\rightarrow 0} \gamma(T) = 0.
$
\end{lemma}
\begin{proof}
See Appendix. 
\end{proof}
\begin{remark}
Lemma~\ref{lemma:estiamte-error-bound} implies that the uncertainty estimation can be made arbitrarily accurate for $t\geq T$, by reducing $T$, the latter only subject to hardware limitations. Additionally, the estimation cannot be arbitrarily accurate for $t\in[0,T)$. This is because the estimate in $[0,T)$ depends on $\tilx(0)$ according to \eqref{eq:adaptive_law}. Considering that $\tilx(0)$ is purely determined by the initial state of the system, $x_0$, and the initial state of the predictor, $\hatx_0$,  it does not contain any information of the uncertainty. Since $T$ is usually very small in practice,  lack of a {\it tight} estimation error bound for the interval $[0,T)$ will not cause an issue from a practical point of view. 
\end{remark}


\subsection{Adaptive Robust QP Formulation}\label{sec:sub-aR-QP}
Let
{\setlength{\mathindent}{0mm} 
\begin{IEEEeqnarray}{rl} 
\Psi_V(t,x,u)  \triangleq & L_fV(x)+L_gV(x)u+  V_x(x)\hd(t), \nonumber \\
& +\norm{V_x(x)}\gamma(T) \label{eq:psi_V-defn}, \\
 \Psi_h(t,x,u)\triangleq & L_fh(x)+L_gh(x)u+ h_x(x)\hd(t), \nonumber \\
    & -\norm{h_x(x)}\gamma(T).\label{eq:psi_h-defn}
\end{IEEEeqnarray}}
We are ready to present the main result in the following theorem. 
\begin{theorem}\label{theorem:h_V_cond_adaptive}
For any $t\geq T$,
\begin{enumerate}
 \item the condition 
    \begin{equation} \label{eq:V-inequality-robust}
 \inf_{u\in U} \Psi_V(t,x,u) \leq  -\alpha(V(x))
 \end{equation}
 is a sufficient condition for \eqref{eq:V-CLF-cond-inf}, and also a necessary condition for \eqref{eq:V-CLF-cond-inf} when $T$ $\rightarrow$ 0.
 \item the condition 
 \begin{equation} \label{eq:h-inequality-robust}
\sup_{u\in U} \Psi_h(t,x,u) \geq -\beta(h(x))
\end{equation}
is a sufficient condition for 
 \eqref{eq:h-CBF-cond-inf}, and also a necessary condition for 
 \eqref{eq:h-CBF-cond-inf} when $T$ $\rightarrow$ 0.

\end{enumerate}
\end{theorem}
\begin{proof} Due to space limit, we only prove 2), while 1) can be proved analogously.  For proving that \eqref{eq:h-inequality-robust} is sufficient for \eqref{eq:h-CBF-cond-inf}, comparison of \eqref{eq:h-CBF-cond-inf} and \eqref{eq:h-inequality-robust} indicates that we only need to show 
$
    h_x(x)d(t,x) \geq h_x(x)\hd(t)-\norm{h_x(x)}\gamma(T)
$ for any $t\geq T$.
The preceding inequality actually holds since 
{\setlength{\mathindent}{0mm}
\begin{align*}
    h_x(x)d(t,x) &=  h_x(x)\left(\hd(t)+d(t,x)-\hd(t)\right)\\
   & \geq h_x(x)\hd(t) - \norm{h_x(x)}\norm{d(t,x)-\hd(t)} \\
   & \geq  h_x(x)\hd(t) - \norm{h_x(x)}\gamma(T), 
\end{align*}}where the last inequality is due to  \eqref{eq:estimation_error_bound}.  For proving  the necessity, we notice that when $T \rightarrow 0$, $\hd(t) \rightarrow d(t,x)$ for any $t\geq T$ according to Lemma~\ref{lemma:estiamte-error-bound}, which implies 
that (LHS) of \eqref{eq:h-CBF-cond-inf} equates (LHS) of \eqref{eq:h-inequality-robust} considering \eqref{eq:psi_h-defn}.
Therefore, \eqref{eq:h-inequality-robust} is also necessary for \eqref{eq:h-CBF-cond-inf} for $t\geq T$ when $T \rightarrow 0$. \qed
\end{proof}

\begin{remark}\label{remark:conservatism-discussion}
Theorem~\ref{theorem:h_V_cond_adaptive} indicates that the condition \eqref{eq:V-inequality-robust} (\eqref{eq:h-inequality-robust}) that depends on the adaptive estimation is always a sufficient condition for \eqref{eq:V-CLF-cond-inf} (\eqref{eq:h-CBF-cond-inf}) that depends on the true uncertainty  for $t\geq T$. Additionally, the conservatism of the condition \eqref{eq:V-inequality-robust} or \eqref{eq:h-inequality-robust} can be arbitrarily reduced by reducing $T$.
\end{remark}

With the inequalities \eqref{eq:V-inequality-robust} and \eqref{eq:h-inequality-robust} that depend on the adaptive estimation of the uncertainty, we can formulate  a robust QP, which we call aR-QP:
{\setlength{\mathindent}{0mm}
\setlength{\belowdisplayskip}{0pt} \setlength{\belowdisplayshortskip}{0pt}
\begin{align}
    \hline
       & u^\star(t,x)  =  \argmin_{(u, \delta)\in \mbR^{m+1}} \frac{1}{2}u^TH(x)u + p\delta^2 \  (\textup{\RaQP}) \nonumber \\ 
     \textup{s.t. } &
     \small
     \Psi_V(t,x,u) + \alpha(V(x)) < \delta, \label{eq:V-inequality-robust-in-qp}  \\
    & 
    \begin{footnotesize}
    \Psi_h(t,x,u)+\beta(h(x)) >0, 
    \end{footnotesize}  \label{eq:h-inequality-robust-in-qp} 
    \\
    & u\in U. \\
    \hline\nonumber
    \vspace{-20mm}
\end{align}}Here, $\Psi_V$ and $\Psi_h$ are defined in \eqref{eq:psi_V-defn} and \eqref{eq:psi_h-defn}, respectively.
\begin{remark}
The  estimation sampling time  $T$ can be different from the sampling time $T_{qp}$ for solving the (\RaQP) problem. Considering the effect of $T$ on the estimation accuracy, as well as the relatively low computational cost of the adaptive estimation compared to solving the (\RaQP) problem, $T$ can be set much smaller than $T_{qp}$.  
\end{remark}
\vspace{-0mm}
\section{Simulation results}\label{sec:simulation}
In this section, we validate the theoretical development using the adaptive cruise control (ACC) problem from \cite{ames2016cbf-tac}.
\vspace{-1mm}
\subsection{ACC Problem Setting}
The lead and following cars are modeled as point masses moving on a straight road. The following car is equipped with an ACC system, and and its objective is to  cruise at a given constant speed, while maintaining a safe following distance as specified by a time headway. Let $v_l$ and $v_f$ denote the speeds (in $\mbox{m/s}$) of the lead car and the following car, respectively, and $D$ be the distance (in $\mbox{m}$) between the two vehicles. Denoting $x=[v_l,v_f,D]$ as the system state, the dynamics of the system can be described as 
{ \begin{equation}
    \begin{bmatrix}
    \dot{v_l} \\
    \dot{v_f}\\
    \dot{D}  
    \end{bmatrix}
     = \underbrace{\begin{bmatrix}
    a_l \\
    0\\
    v_l-v_f
    \end{bmatrix}}_{f(x)}
    + \underbrace{\begin{bmatrix}
    0\\
    1/m \\
    0
    \end{bmatrix}}_{g(x)} u +
    \underbrace{\begin{bmatrix}
    0 \\
   -F_r/m +d_0(t) \\
    0
    \end{bmatrix}}_{d(t,x)},
\end{equation}}where $u$ and $m$ are the control input and the mass of the following car, respectively, $a_l$ is the acceleration of the lead car, $F_r = f_0+f_1v_f + f_2 v_f^2$ is the aerodynamic drag term with unknown constants $f_0, f_1$ and $f_2$, $d_0(t)$ is external disturbance (reflecting the unmodeled road condition or aerodynamic force).

The safety constraint requires the following car to keep a safe distance from the lead car, which can be expressed as $D/v_f\geq \tau_d$ with $\tau_d$ being the desired time headway. Defining the function $h=D-\tau_d v_f$, the safety constraint can be expressed as $h\geq 0$. In terms of control objective,  the following car should achieve a desired speed $v_d$ when adequate headway is ensured. This objective naturally leads to the CLF, $V = (v_f- v_d)^2$.

The parameters used in the simulation are shown in \cref{table:acc-parameters}, where $m, f_0, f_1, f_2, \tau_d$ are selected following \cite{ames2016cbf-tac}. The disturbance is set to $d_0(t) = 0.2 g\sin(2\pi 10t)$ following \cite{xu2015robustness}. The input constraints are set to $-u_{\max} \leq u \leq u_{\max}$, where $u_{\max} = 0.4\ mg$ with $g$ being the gravitational constant.



{\setlength{\tabcolsep}{2pt}
\vspace{-2mm}
\begin{table}[htb]
\caption{Parameter Settings}\label{table:acc-parameters}
\begin{tabular}{ll|ll|ll|ll}
\toprule
$g$   & 9.81 $\mbox{Ns}^2\mbox{/m}$ & $f_2$    & 0.25 $\mbox{Ns}^2\mbox{/m}$ & $\tau_d$ & 1.8 $\mbox{s}$     & $p$         & 100                               \\
$m$   & 1650 $\mbox{kg}$     & $v_d$    & 22 $\mbox{m/s}$      & $T_{qp}$      & 0.01 $\mbox{s}$   & $\alpha(V)$    & 5$V$                  \\
$f_0$ & 0.1 $\mbox{N}$       & $x(0)$   & [18 12 80]           &$T$    & 1 $\mbox{ms}$  &  $\beta(h)$ &  $h$    \\
$f_1$ & 5 $\mbox{Ns/m}$      & $u_{\max}$ & 0.4 $mg$              & $a$         & 1          &             &     \\    \bottomrule              
\end{tabular}
\vspace{-3mm}
\end{table}
}

\subsection{Uncertainty Estimation Setting}
According to \eqref{eq:gammaTs-defn}, given the estimation sampling time $T$,  one would like to obtain the smallest values of the constants in Assumption~\ref{assump:lipschitz-bound-fg} to get the tightest estimation error bound, $\gamma(T)$. For this ACC problem,  the constants in Assumption~\ref{assump:lipschitz-bound-fg} are selected as
 {\setlength{\mathindent}{0pt} 
\begin{align*}
  &  l_t = (0.2g(2\pi)10)\xi,\ l_d = (f_1+2f_2v_{\max})\xi,\ b_d = (0.2g) \xi  , 
\end{align*}}where $v_{\max}= 160\ \textup{km/h}$ is the maximum speed considered,  $\xi\geq 1$ is a constant to reflect the conservatism in estimating the constants $l_t, l_d$  and $b_d$
that satisfy \cref{eq:tilf-lipschitz-cond,eq:tilf-tilg-x0-bound} in \cref{assump:lipschitz-bound-fg}. We set $\xi=2$. 
We further set $a=1$ in \eqref{eq:state-predictor}. With this setting, the estimation error bounds under different $T$ are computed according to \cref{eq:gammaTs-defn} and  listed in \cref{table:gammaTs-vs-Ts-acc}. For the simulation results in \cref{sec:simu_results_delta_tx}, $T=1$~ms is selected. 
\vspace{-5mm}
\begin{table}[htb]
\caption{$\gamma(T)$ versus $T$}\label{table:gammaTs-vs-Ts-acc}
    \centering
    \begin{tabular}{lllrl}
    \toprule
$T$         & 10 ms & 1 ms  & 0.1 ms & 0.01 ms \\
$\gamma(T)$ & 2.98  & 0.298 & 0.0298 & 0.00298 \\
\bottomrule
\end{tabular}
\vspace{-5mm}
\end{table}
\subsection{QP Setting}
We consider several QP controllers:
\begin{itemize}
    \item A standard QP controller defined by (\QP) using the true uncertainty $d(t,x)$, 
    \item A  standard QP controller ignoring the uncertainty, obtained by setting $d(t,x)\equiv 0 $ in (\QP),
    \item A robust QP (R-QP) controller proposed in \cite{nguyen2016optimalACC} using the worst-case bound on $d(t,x)$ in \eqref{eq:d-xdot-bound}, which was obtained by setting $\hat d(t) \equiv0$ in (\RaQP), and
    \item A adaptive robust QP (aR-QP) controller from (\RaQP). 
\end{itemize}
Note that the first controller is not implementable due to its reliance on the true uncertainty model (see Remark~\ref{remark:qp-not-implementable}), and is included to merely show the {\it ideal} performance. 
The objectives in both (\QP) and (\RaQP) are set to be 
$
\min \frac{1}{2}Hu^2 + \frac{1}{2}p \delta^2,
$
where $H = 1/m^2$, $p=100$. We further set $\alpha(V) = 5 V$ and $\beta(h) = h$. Under the above settings, the conditions \eqref{eq:V-CLF-cond-inf-R} and \eqref{eq:h-CBF-cond-inf-R} in Lemma~\ref{lem:clf-cbf-verifiable-defn} can be verified, which indicates that $V$ and $h$ are indeed a CLF and CBF for the uncertain system, respectively. We set $T_{qp}=0.01$ second. 
\subsection{Simulation Results}\label{sec:simu_results_delta_tx}
The results are shown in Fig.~\cref{fig:xt_uncert_xu,fig:xt_uncert_h,fig:xt_uncert_hd}. As expected, the QP controller using the true uncertainty $d(t,x)$ achieved good performance in tracking the desired speed when enforcing safety was not a major concern, while maintaining the safety ($h$ was above $0$) throughout the simulation with minimal conservatism ($h$ was quite close to $0$ when enforcing safety was a major concern). On the other hand, the QP controller ignoring the uncertainty did not provide satisfactory tracking performance (note the relatively large tracking error between $25$ and $32$ seconds in Fig.~\ref{fig:xt_uncert_xu}) even when there was adequate headway; it also failed to guarantee the safety throughout the simulation ($h$ was below $0$ during some intervals). Although the R-QP controller did provide safety guarantee as shown by the trajectory of $h$ in Fig.~\ref{fig:xt_uncert_h}, it yielded rather conservative performance: (1) $h$ was constantly far way from $0$; (2) speed tracking objective was often compromised more than necessary (note the speed decrease between $2$ and $5$ seconds).
\begin{figure}[htb]
\vspace{-5mm}
    \centering
    \includegraphics[width=.95\columnwidth]{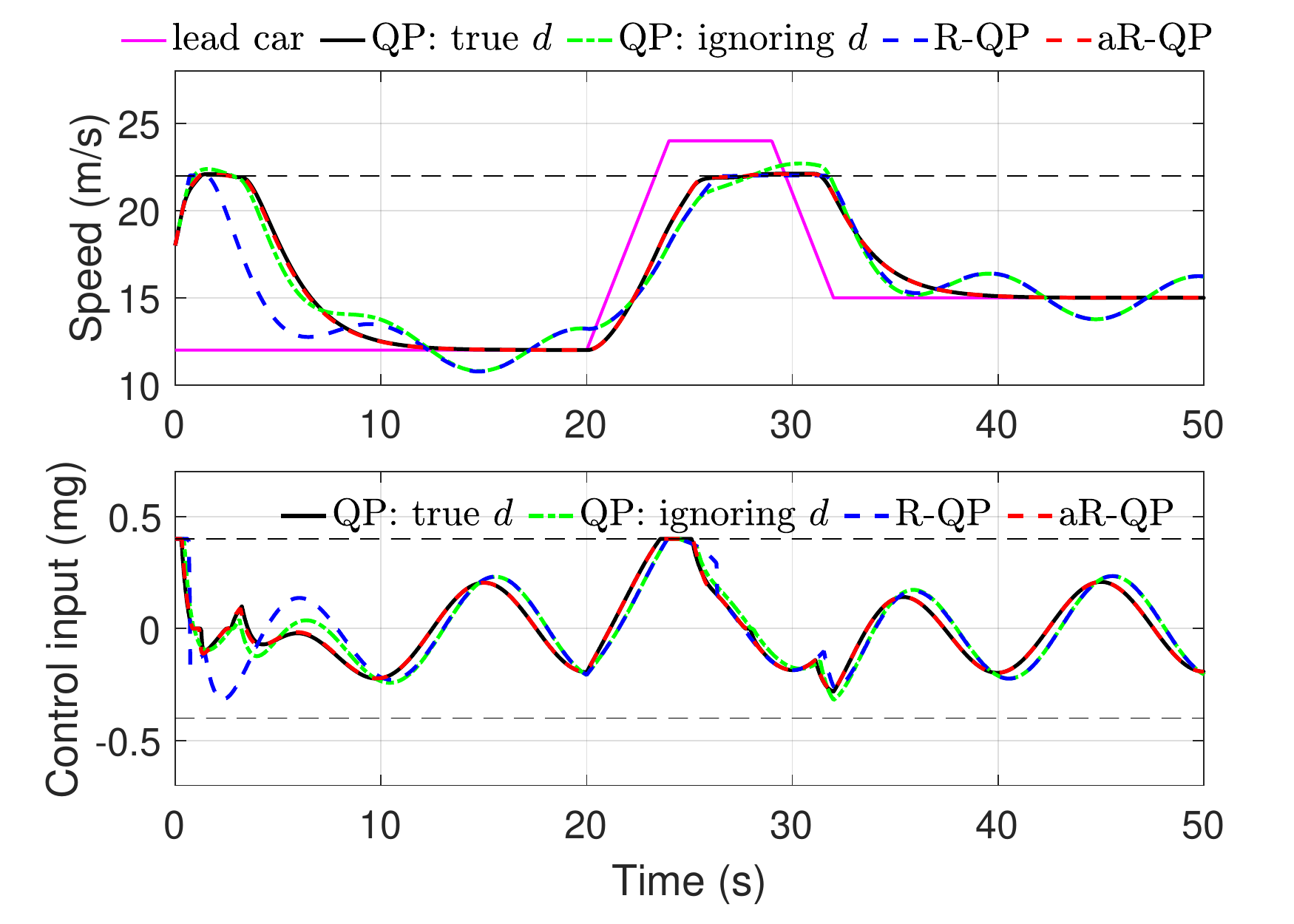}
    \vspace{-2mm}
     \caption{Trajectories of  the following car speed  (top) and control input (bottom). Black dashed lines denote the desired speed (top) and input limits (bottom).}
    \label{fig:xt_uncert_xu}
    \vspace{-3mm}
\end{figure}

Finally, utilizing the estimated uncertainty, 
the aR-QP controller almost recovered the performance of the QP controller using the true uncertainty, in terms of both the car speed and control input. It also maintained the safety throughout the simulation with slightly increased conservatism compared to the performance of the QP controller using the true uncertainty, as shown in Fig.~\ref{fig:xt_uncert_h}. Figure~\ref{fig:xt_uncert_hd} depicts the trajectories of the true and estimated uncertainties with error bounds also displayed. One can see that the estimated uncertainty overlaps the true uncertainty after a {\it single}  sampling interval, $T$. Additionally, the true uncertainty always lies within a tube determined by the estimated uncertainty and the error bounds defined in \eqref{eq:estimation_error_bound}, which is consistent with Lemma~\ref{lemma:estiamte-error-bound}.
\begin{figure}[htb]
\vspace{-2mm}
    \centering
    \includegraphics[width=.95\columnwidth]{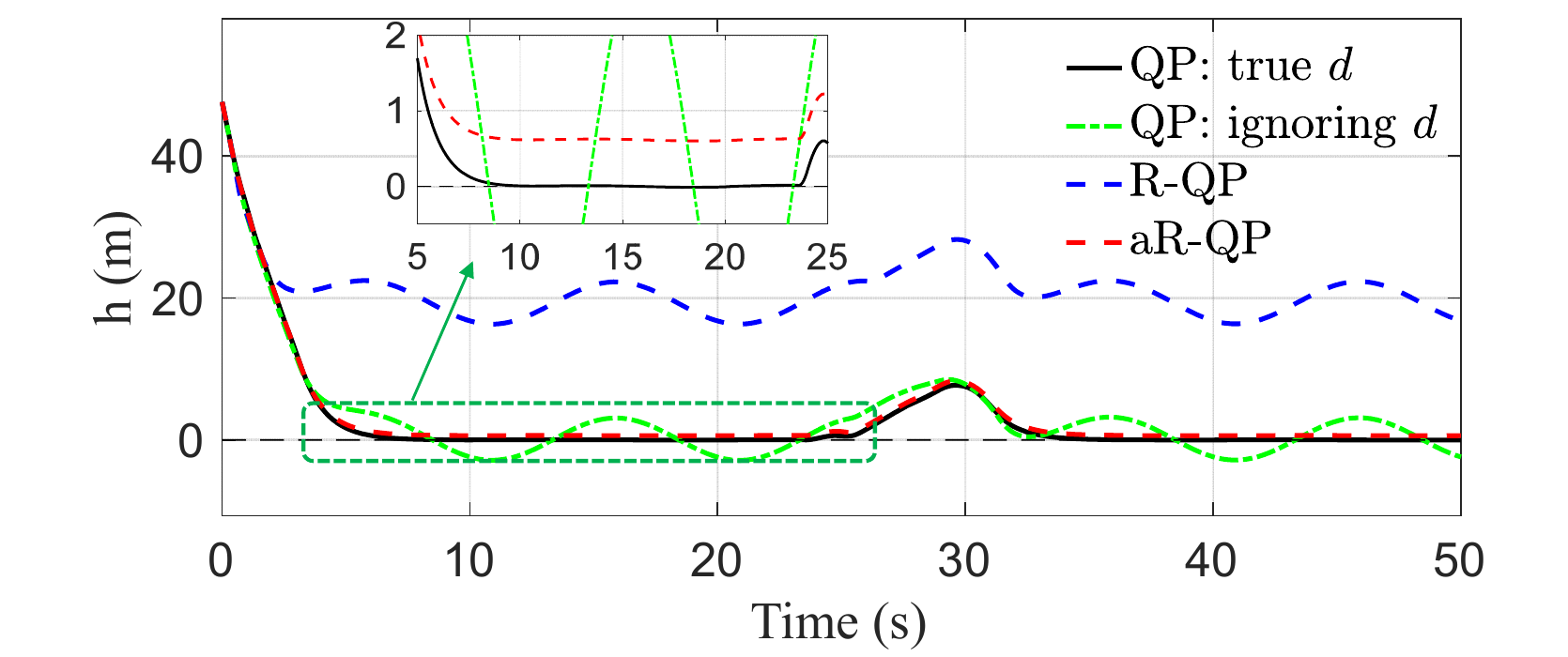}
    \vspace{-2mm}
    \caption{Trajectories of the barrier function}
    \label{fig:xt_uncert_h}
    \vspace{-3mm}
\end{figure}
\begin{figure}[htb]
 \vspace{-5mm}
    \centering
    \includegraphics[width=0.95\columnwidth]{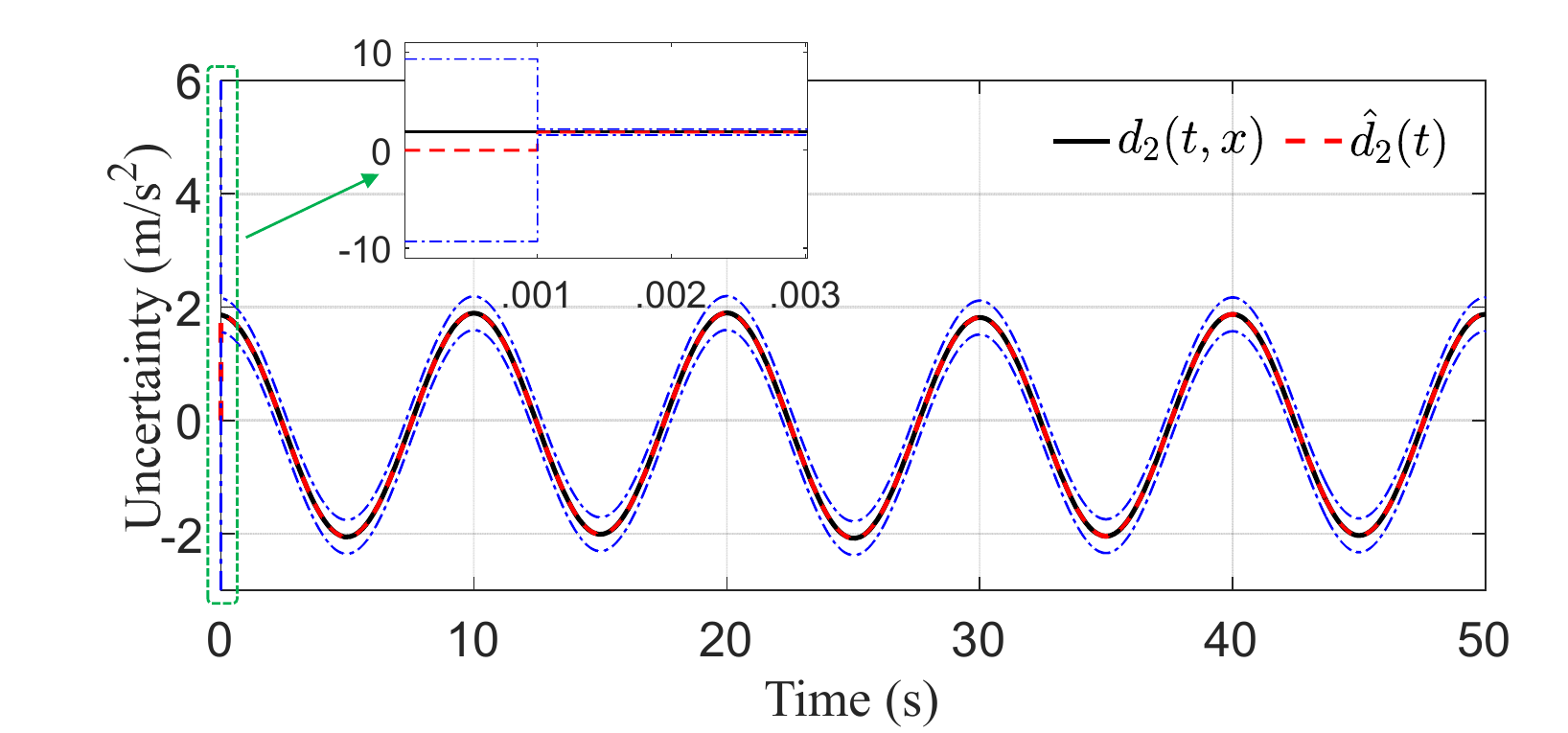}
    \vspace{-2mm}
    \caption{Trajectories of the estimated and true uncertainties. Blue dash-dotted lines denote the estimation error bounds computed according to \eqref{eq:estimation_error_bound}. $d_2(t,x)$ ($\hat d_2(t)$) is the second element of  $d(t,x)$ ($\hat d(t)$).}
    \label{fig:xt_uncert_hd}
    \vspace{-5mm}
\end{figure}

\section{CONCLUSIONS}\label{sec:conclusion}

This paper summarizes an adaptive estimation inspired approach to control Lyapunov and barrier functions based control via QPs in the presence of time-varying and state-dependent uncertainties. An adaptive estimation law is proposed to estimate the pointwise value of the uncertainties with pre-computable error bounds. The estimated uncertainty and the error bounds are then used to formulate a robust QP, which ensures that the actual uncertain system will not violate the safety constraints. It is also shown both theoretically and numerically that the estimation error bound and the conservatism of the robust constraints can be systematically reduced by reducing the estimation sampling time. The proposed approach is validated on an adaptive cruise control problem through comparisons with existing approaches.  




\appendix

{\it Proof of Lemma~\ref{lemma:estiamte-error-bound}:}  
From \eqref{eq:dynamics} and \eqref{eq:state-predictor}, the prediction error dynamics are obtained as
\begin{equation}
    \dot{\tilx}=-a\tilx + \hd(t)-d(t,x), \quad \tilx(0) = 0. 
\end{equation}
Therefore, $\hd(t) = 0 $ for any $t\in[0,T)$ according to \eqref{eq:adaptive_law}. Further considering the bound on $d$ in \eqref{eq:d-xdot-bound}, we have
\begin{equation}\label{eq:estimation-error-bound-0-T}
    \norm{\hd(t)-d(t,x)} \leq \theta, \quad \forall t\in[0,T).
\end{equation}
We next derive the bound on $\norm{\hd(t)-d(t,x)}$ for $t\geq T$. For notation brevity, we often write $d(t,x(t))$ as $d(t)$ hereafter. For any $t\in [iT, (i+1)T)$ ($i\in \mbZ_0$), we have
{\setlength{\mathindent}{0mm}
$$
    \tilx(t) = e^{-a(t-iT)}\tilx(iT) + \int_{iT}^t e^{-a(t-\tau)}(\hd(
    \tau)-d(\tau))d\tau.
$$}
Since $\tilx(t)$ is continuous, the preceding equation implies 
{\setlength{\mathindent}{0mm} \small
\begin{align}
  &  \tilx((i+1)T) \nonumber = ~  e^{-aT}\tilx(iT) +\int_{iT}^{(i+1)T} e^{-a((i+1)T-\tau)}d\tau \hd(iT) \nonumber \\
    & \hspace{2cm}  - \int_{iT}^{(i+1)T} e^{-a((i+1)T-\tau)}d(\tau)d\tau. \nonumber \\
     = &~e^{-aT}\tilx(iT) + \frac{1-e^{-aT}}{a}\hd(iT)   - \int_{iT}^{(i+1)T} e^{-a((i+1)T-\tau)}d(\tau)d\tau  \nonumber\\
    =& ~ - \int_{iT}^{(i+1)T} e^{-a((i+1)T-\tau)}d(\tau)d\tau,  \label{eq:tilx-iplus1-Ts}
\end{align}
 where the first and last equalities are due to the estimation law \eqref{eq:adaptive_law}.} 

Since $x(t)$ is continuous,
$d(t,x)$ is also continuous given Assumption~\ref{assump:lipschitz-bound-fg}.  Furthermore, considering that  $e^{-a((i+1)T-\tau)}$ is always positive, we can apply the first mean value theorem in an element-wise manner\footnote{{Note that the mean value theorem for definite integrals 
 only holds for scalar valued functions.}} to \eqref{eq:tilx-iplus1-Ts}, which leads to 
{\setlength{\mathindent}{0mm}
\begin{align}
    \tilx((i+1)T) = & -\int_{iT}^{(i+1)T} e^{-a((i+1)T-\tau)}d\tau \bbracket{d_j(\tau_j^*)} \nonumber\\
    =& -\frac{1}{a}(1-e^{-aT})\bbracket{d_j(\tau_j^*)}, \label{eq:xtilde-iplus1-Ts-d-tau}
\end{align}}for some $\tau_j^*\in (iT, (i+1)T)$ with $j\in\mbZ_1^n$ and $i\in \mbZ_0$, where $d_j(t)$ is the $j$-th element of $d(t)$, and
{$$\bbracket{d_j(\tau_j^*)}\triangleq 
[d_1(\tau_1^*),\cdots, d_n(\tau_n^*)
]^\top.$$}The adaptive law \eqref{eq:adaptive_law} indicates that for any $t$ in $[(i+1)T, (i+2)T)$, we have
$\hd(t) = -\frac{a}{e^{aT-1}}\tilx((i+1)T).$
The preceding equality and \eqref{eq:xtilde-iplus1-Ts-d-tau} imply that for any $t$ in $[(i+1)T, (i+2)T)$ with $i\in \mbZ_0$, there exist $\tau_j^*\in(iT,(i+1)T)$ ($j\in\mbZ_1^n$) such that 
\begin{equation}\label{eq:hatd-d-tau-star-relation}
    \hd(t) = e^{-aT}\bbracket{d_j(\tau_j^*)}.
\end{equation}
Note that 
{ 
\begin{align}
 &    \norm{d(t))-\bbracket{d_j(\tau_j^*)}} \leq \sqrt{n}\infnorm{d(t)-\bbracket{d_j(\tau_j^*)}} \nonumber\\
    = & \sqrt{n}\abs{d_{\bar j_t}(t)-{d_{\bar j_t}(\tau_{\bar j_t}^*)}} 
    \leq \sqrt{n}\norm{d(t)-{d(\tau_{\bar j_t}^*)}},\label{eq:d-dStar-index-conversion}
\end{align}
}where $\bar j_t=\arg\max_{j\in\mbZ_1^n} \abs{d_j(t)-{d_j(\tau_j^*)}}$.  Similarly,
{\setlength{\mathindent}{0mm} \small 
\begin{align}
 \norm{\bbracket{d_j(\tau_j^*)}} \leq& \sqrt{n}\infnorm{\bbracket{d_j(\tau_j^*)}} 
     =  \sqrt{n}\abs{{d_{{\bar j}}(\tau_{{\bar j}}^*)}} \leq \sqrt{n}\norm{{d(\tau_{{\bar j}}^*)}},  \label{eq:dStar-index-conversion}
\end{align}}where ${\bar j}=\arg\max_{j\in\mbZ_1^n} \abs{{d_j(\tau_j^*)}}$. 
Therefore, for any $t\in[(i+1)T, (i+2)T)$ ($i\in \mbZ_0$), we have
{\begin{align}
  &  \norm{d(t)-\hd(t)}     = \norm{d(t) - e^{-aT}[d_j(\tau_j^*)]} \nonumber\\
\leq & \norm{d(t) -[d_j(\tau_j^*)]}  + (1-e^{-aT})\norm{[d_j(\tau_j^*)]} \nonumber \\ 
\leq & \sqrt{n}\norm{d(t)-{d(\tau_{\bar j_t}^*)}} +(1-e^{-aT})\sqrt{n}\norm{{d(\tau_{{\bar j}}^*)}},
\label{eq:d-sigmahat-bound}
\end{align}}for some $\tau_{\bar j_t}^*\in (iT,(i+1)T)$ and  
$\tau_{{\bar j}}^*\in (iT,(i+1)T))$, where the equality is due to \eqref{eq:hatd-d-tau-star-relation}, and the last inequality is due to \eqref{eq:d-dStar-index-conversion} and \eqref{eq:dStar-index-conversion}. The inequality \eqref{eq:d-xdot-bound} implies that 
{\setlength{\mathindent}{0cm}\small
\begin{align*}
    \norm{x(t)-x(\tau_{\bar j_t}^*)} & \leq \int_{\tau_{\bar j_t}^*}^t \norm{\dot{x}(
   \tau)}d\tau 
   \leq \int_{\tau_{\bar j_t}^*}^t\phi d\tau = \phi(t-\tau_{\bar j_t}^*).
\end{align*}}The preceding inequality and \eqref{eq:tilf-lipschitz-cond} indicate that 
{\setlength{\mathindent}{0cm}
\small
\begin{align}
 &   \norm{d(t,x(t))   - d(\tau_{\bar j_t}^*,x(\tau_{\bar j_t}^*))} 
   \leq l_t(t-\tau_{\bar j_t}^*) + l_d\norm{x(t)-x(\tau_{\bar j_t}^*)} \nonumber \\
     \leq &  (l_t+l_d\phi)(t-\tau_{\bar j_t}^*)
   =  \eta (t-\tau_{\bar j_t}^*) \leq 2\eta T, \label{eq:d-t-taustar-bound}
\end{align}}where $\eta$ is defined in \eqref{eq:beta-defn}, and the last inequality is due to the fact that $t\in[(i+1)T, (i+2)T)$ and  $\tau_{\bar j_t}^*\in (iT, (i+1)T)$. 

Finally, plugging \eqref{eq:d-t-taustar-bound} into \eqref{eq:d-sigmahat-bound} leads to 
{\setlength{\mathindent}{0mm}\small
\begin{align}
   & \norm{d(t,x(t))-\hd(t)} \nonumber 
    \leq  2\sqrt{n}\eta T +\sqrt{n} (1- e^{-aT})\norm{d(\tau_{{\bar j}}^*,x(\tau_{{\bar j}}^*))} \nonumber \\
   & \leq  2\sqrt{n}\eta T +\sqrt{n} (1- e^{-aT})\theta = \gamma(T), \quad \forall t\geq T, \label{eq:estimation-error-bound-T-inf}
\end{align}}where the second inequality is due to \eqref{eq:d-xdot-bound}. From \eqref{eq:estimation-error-bound-0-T}
and \eqref{eq:estimation-error-bound-T-inf}, we arrive at \eqref{eq:estimation_error_bound}.
Considering that $X$ and $U$ are compact, the constants $\theta$ (defined in \eqref{eq:theta-defn}), $\phi$ (defined in \eqref{eq:phi-defn})  and $\eta$ (defined in \eqref{eq:beta-defn}) are all finite, the definition of $\gamma(T)$ in \eqref{eq:gammaTs-defn} immediately implies    $\lim_{T\rightarrow 0} \gamma(T) = 0$. \qed

\bibliographystyle{ieeetr}
\bibliography{refs-acrl,refs-pan,refs}

\end{document}